\documentclass[11pt]{amsart}
\usepackage[linkcolor=blue, citecolor = blue]{hyperref}
\hypersetup{colorlinks=true}

\usepackage{amsthm,amsfonts,amsmath,amssymb,amscd} 

\usepackage{verbatim}
\usepackage{float}
\usepackage{wrapfig}
\usepackage{mathtools}
\usepackage[color=orange!20, linecolor=orange]{todonotes}

\usepackage{enumitem}

\usepackage{ytableau}

\usepackage{tikz}
\usetikzlibrary{decorations.markings}
\tikzstyle{vertex}=[circle, draw, inner sep=0pt, minimum size=4pt]

\usetikzlibrary{arrows,automata}
\usepackage{tikz, tikz-3dplot, pgfplots}
\usepackage{tkz-graph}
\usetikzlibrary[positioning,patterns]
\usepackage{wasysym}

\newtheorem{theorem}{Theorem} %[section]

\newtheorem{lemma}[theorem]{Lemma}
\newtheorem{corollary}[theorem]{Corollary}

\theoremstyle{definition}

\theoremstyle{remark}
\newtheorem{remark}{Remark}

\title[Dual Grothendieck polynomials via last-passage percolation]{
Dual Grothendieck polynomials via last-passage percolation
} 

\author[Damir Yeliussizov]{Damir Yeliussizov}
\address{KBTU, Almaty, Kazakhstan}
\email{\href{mailto:yeldamir@gmail.com}{yeldamir@gmail.com}}

\begin{document}
\begin{abstract}
The ring of symmetric functions has a basis of dual Grothendieck polynomials that are inhomogeneous $K$-theoretic deformations of Schur polynomials. We prove that dual Grothendieck polynomials determine 
column distributions for a directed last-passage percolation model. 
\end{abstract}
\maketitle
%\tableofcontents
\section{Introduction}
In this note we show a surprising %conceptual 
connection between  
(i) the dual Grothendieck polynomials that are deformations of Schur polynomials arising in $K$-theoretic Schubert calculus, 
and (ii) a directed last-passage percolation model (which can also be viewed as the {corner growth model} or a totally asymmetric simple exclusion process).

\subsection{Dual Grothendieck polynomials} The ring %$\Lambda$ 
of symmetric functions has an inhomogeneous basis $\{g_{\lambda} \}$ called the {\it dual  Grothendieck polynomials}. The symmetric polynomials $g_{\lambda}(x_1, \ldots, x_n)$ can be defined via the following combinatorial formula
$$
g_{\lambda}(x_1, \ldots, x_n) := \sum_{\pi\, :\, \mathrm{sh}(\pi) = \lambda} \prod_{i = 1}^{n} x_i^{c_i(\pi)},
$$
where the sum runs over {\it plane partitions} $\pi$ of shape $\lambda$ with largest entry at most $n$ and $c_i(\pi)$ is the number of columns of $\pi$ containing $i$. It is easy to see that $g_{\lambda} = s_{\lambda} + \text{lower degree terms}$, where $s_{\lambda}$ is the Schur polynomial.
This basis was explicitly introduced and studied in \cite{lp} (and earlier implicitly in \cite{lenart, buch}) in relation to the $K$-theory of Grassmannians.  %Alternatively, $g_{\lambda}$ can be computed via Jacobi-Trudi-type determinantal formulas \cite{dy}, see \eqref{jt}.  
More properties of these functions can also be found in \cite{dy, dy2}.

\subsection{Directed last-passage percolation} %The corner growth model}
Let $W = (w_{ij})_{i,j\ge 1}$ be a random matrix with independent entries $w_{ij}$ that have geometric distribution with parameters $q_j \in (0,1)$, i.e. 
$$\mathrm{Prob}(w_{ij} = k)  = (1 - q_j)\, q_j^k, \quad k \in \mathbb{N}.$$ 
A lattice path $\Pi$ with vertices in $\mathbb{N}^2$ is called a {\it directed path} if it uses only steps of the form $(i,j) \to (i+1, j), (i,j+1)$. Define the {\it last-passage times} $G(m,n)$ as follows:
$$
G(m,n) := \max_{\Pi : (1,1) \to (m,n)} \sum_{(i,j) \in \Pi} w_{ij},
$$ 
where the maximum is over directed paths $\Pi$ from $(1,1)$ to $(m,n)$. 
The function $G$ presents certain random growth. This probabilistic model, which can also be viewed as the corner growth model or a totally asymmetric simple exclusion process (TASEP), was studied intensively (especially in the {\it iid} case $q_j = q$), see \cite{joh1, bar,  joh2, sep, romik} and references therein. Let us call the matrix $G = (G(m,n))_{m,n\ge 1}$ as the {\it percolation matrix}. 

\subsection{Column distributions of the percolation matrix} 
Our main result is the formula showing that joint distribution of elements along any column in the percolation matrix $G$ is proportional to evaluations of dual Grothendieck polynomials.
Let 
$$
\mathcal{P}_{m} := \{\lambda = (\lambda_1, \ldots, \lambda_m) \in \mathbb{N}^m : \lambda_1 \ge \cdots \ge \lambda_m \ge 0 \}
$$
be the set of integer partitions with at most $m$ parts.

\begin{theorem}\label{main}
Let $\lambda = (\lambda_1, \ldots, \lambda_m) \in \mathcal{P}_m$. 
The following formula holds 
$$
\mathrm{Prob}(G({m,n}) = \lambda_1, \ldots, G(1,n) = \lambda_m) = \prod_{i = 1}^{n}(1 - q_i)^m\, g_{\lambda}(q_1, \ldots, q_n). %\leqno{(\leftmoon)} 
$$
\end{theorem}

On one hand, this formula can be viewed as a natural probabilistic interpretation of dual Grothendieck polynomials. On the other hand, it can also be used for computing distribution formulas in the percolation matrix. We prove Theorem~\ref{main} combinatorially, using certain bijection between plane partitions and integer matrices. We then give some applications. %of this formula. %using it in both directions. 
For example, we present new generating function identities for dual Grothendieck polynomials and determinantal formulas for distributions of the percolation matrix.  

\section{Proof of the main theorem}
\subsection{Plane partitions and $\mathbb{N}$-matrices}
An {\it $\mathbb{N}$-matrix} is a matrix of nonnegative integers with only finitely many nonzero elements. A {\it plane partition} is an $\mathbb{N}$-matrix $\pi = (\pi_{ij})_{i, j \ge 1}$ such that 
$$
\pi_{ij} \ge \pi_{i+1\, j}, \quad \pi_{i j} \ge \pi_{i\, j+1}, \quad i,j \ge 1.
$$
The {\it shape} of $\pi$ is defined as %the {\it Young diagram} 
$\mathrm{sh}(\pi) := \{(i,j) : \pi_{i j} > 0 \}$. 

\vspace{0.5em}

Given a plane partition $\pi$, define the {\it descent level sets} 
\begin{align*} %\label{ldes}
D_{i \ell} := \{ j : \pi_{i j} = \ell  > \pi_{i+1\, j} \},
%D_{i j} := \{ \ell : \pi_{j \ell} = i > \pi_{j+1\, \ell} \},
\end{align*}
i.e. $D_{i \ell}$ is the set of column indices of the entry $\ell$ in $i$th row of $\pi$ that are strictly larger than the entry below. 
Let $d_{i \ell} := |D_{i \ell}|$ and $D := (d_{i \ell})_{i, \ell \ge 1}$.

\vspace{0.5em}

Define the map $\Phi : \text{\{plane partitions\}} \to \{\mathbb{N}\text{-matrices}\}$ by setting 
\begin{align*}%\label{piq}
\Phi(\pi) = D.
\end{align*}
For example,
{\small
\begin{center}
$\Phi : $
\ytableausetup{aligntableaux = center}
\begin{ytableau}
 4 & {4} & {2} \\
 {4} & 2 & {1} \\  %\textcolor{blue}
 {2} & {2}
\end{ytableau}
$~\longmapsto
\left(
\begin{matrix}
0 & 1 & 0 & 1\\
1 & 0 & 0 & 1\\
0 & 2 & 0 & 0\\
%1 & 1 & 0\\
\end{matrix}
\right)
$
\end{center}
}

In fact, $\Phi$ is a bijection; we can uniquely reconstruct $\pi$ given the matrix $D$. We refer to \cite{dy4,dy5} for more on this bijection. Denote by $\mathrm{PP}(m,n)$ the set of plane partitions with at most $m$ rows and largest entry at most $n$. In particular, if $\pi \in \mathrm{PP}(m,n)$, then $\mathrm{sh}(\pi) \in \mathcal{P}_m$ and $D = \Phi(\pi)$ has at most $m$ rows and $n$ columns that are nonzero.

\begin{lemma}\label{gen}
Let $W = (w_{i \ell})_{i,\ell = 1}^{m,n}$ be an $m \times n$ matrix, where $w_{i \ell}$ are independent geometrically distributed random variables with parameters $q_\ell$. Let $\pi \in \mathrm{PP}(m,n)$.  
Then 
$$\mathrm{Prob}(W = \Phi(\pi))= \prod_{\ell = 1}^{n} (1 - q_\ell)^m q_\ell^{c_\ell(\pi)} %\mathrm{P}_{b,c}(\pi).
$$
where $c_\ell(\pi)$ is the number of columns of $\pi$ containing $\ell$.
\end{lemma}
\begin{proof}
Let $(d_{i \ell})= \Phi(\pi)$, i.e. $d_{i \ell} = |\{j: \pi_{ij} = \ell > \pi_{i+1 j} \}|$ for $i \in [1,m], \ell \in [1,n]$. Since the entries of $W$ are independent we obtain that 
$$
\mathrm{Prob}(W = \Phi(\pi)) = \prod_{i = 1}^{m} \prod_{\ell = 1}^{n} \mathrm{Prob}(w_{i \ell} = d_{i \ell}) = \prod_{i = 1}^m \prod_{\ell = 1}^n (1 - q_{\ell})\, q_{\ell}^{d_{i \ell}} =  \prod_{\ell = 1}^{n} (1 - q_\ell)^m q_\ell^{c_\ell(\pi)} %\mathrm{P}_{b,c}(\pi)
$$
as $\sum_{i} d_{i \ell} = c_\ell(\pi)$, as needed.
\end{proof}

\begin{lemma}\label{one} %[Properties of $\Phi$]
Let $\pi \in \mathrm{PP}(m,n)$ %be a plane partition 
and $\Phi (\pi)= D  =(d_{i \ell})$. 
Let $\lambda = \mathrm{sh}(\pi)$ be the shape of $\pi$. We have for all $k \in [1,m]$
\begin{align}\label{lamm}
\lambda_k  = \max_{\Pi : (k, 1) \to (m, n)} \sum_{(i,\ell) \in \Pi} d_{i \ell},
\end{align}
where the maximum is over directed paths %(i.e. with steps $(x,y) \to (x+1, y), (x, y+1)$) 
$\Pi$ from $(k, 1)$ to $(m,n)$. %, if $D$ has 
\end{lemma}
\begin{proof}
Take an arbitrary directed path $\Pi$ from $(k,1)$ to $(m,n)$. Then the descent level sets $D_{i \ell}$ for $(i,\ell) \in \Pi$ are pairwisely disjoint. Using this property and since $i \ge k$ for all $(i,\ell) \in \Pi$, we obtain
\begin{equation}\label{lal}
\sum_{(i,\ell) \in \Pi} d_{i \ell} = \sum_{(i,\ell)} |\{ j : \pi_{i j} = \ell > \pi_{i + 1 j} \}| \le \lambda_k.
\end{equation}
On the other hand, suppose the $k$-th row of $\pi$ has entries $(\ell_1 \ge \cdots \ge \ell_{s} > 0)$ where $s = \lambda_k$. Assume the entries $\ell_1, \ldots, \ell_s$ end in rows $i_1 \ge \cdots \ge i_s$ of $\pi$. Then there is a directed path $\Pi$ from $(k,1)$ to $(m,n)$ containing all points $(i_s, \ell_s), \ldots, (i_1, \ell_1)$. The weight of any such path is at least 
$\sum_{j} d_{i_j \ell_j} \ge s = \lambda_k$. Combining this with the inequality \eqref{lal} we obtain \eqref{lamm}. 
\end{proof}

\subsection*{Proof of Theorem~\ref{main}} As we are interested only in joint distribution of the last-passage times $(G(m,n), G(m-1, n), \ldots, G(1,n))$ we can restrict the source random matrix $W$ to the first $m$ rows and $n$ columns. Consider $w_{ij}$ as geometric with parameter $q_{n - j +1}$. By rotation symmetry it is obvious that the corresponding last-passage times produced from the matrix $W' = (w_{m - i + 1, n - j + 1})_{i,j = 1}^{m,n}$ have the same distributions. Now, by definition of $g_{\lambda}$ and using Lemma~\ref{gen} we obtain that
\begin{align*}
\prod_{\ell = 1}^{n}(1 - q_{\ell})^m\, g_{\lambda}(q_1, \ldots, q_n) 
	&= \sum_{\pi \in \mathrm{PP}(n,m),\, \mathrm{sh}(\pi) =\lambda} \prod_{\ell = 1}^n (1 - q_\ell)^m q_\ell^{c_{\ell}(\pi)}\\ 
	&= \sum_{\pi \in \mathrm{PP}(n,m),\, \mathrm{sh}(\pi) =\lambda} \mathrm{Prob}(W' = \Phi(\pi))
\end{align*}
By Lemma~\ref{one} if $W' = \Phi(\pi)$ and $\mathrm{sh}(\pi) = \lambda$, we have
$G(m,n) = \lambda_1, \ldots, G(1,n) = \lambda_m$. Therefore,
$$
\sum_{\pi \in \mathrm{PP}(n,m),\, \mathrm{sh}(\pi) =\lambda} \mathrm{Prob}(W' = \Phi(\pi)) = \mathrm{Prob}(G(m,n) = \lambda_1, \ldots, G(1,n) = \lambda_m)
$$
and hence the result is proved. %$\qed$

\vspace{0.5em}

Next we discuss some applications.

\section{Last-passage time distributions}
\subsection{Parameter symmetry} Since the polynomials $g_{\lambda}$ are symmetric we obtain
\begin{corollary} The distribution
$$
\mathrm{Prob}(G(m,n) = \lambda_1, \ldots, G(1,n) = \lambda_m) 
$$
is invariant under permutations of the parameters $(q_1, \ldots, q_n)$.
\end{corollary}

\subsection{Formulas for last-passage distributions}
First, using branching formulas for $g_{\lambda},$ we easily obtain the following distribution formula as well.
\begin{corollary}\label{le}
We have
\begin{align*}
\mathrm{Prob}(G({m,n}) \le \lambda_1, \ldots, G(1,n) \le \lambda_m) = \prod_{i = 1}^{n}(1 - q_i)^m\, g_{\lambda}(1, q_1, \ldots, q_n). 
\end{align*}
\end{corollary}
\begin{proof}
Theorem~\ref{main} combined with the branching relation
$$
\sum_{\mu \subset \lambda} g_{\mu}(q_1, \ldots, q_n)
= g_{\lambda}(1,q_1, \ldots, q_n)
$$
easily imply the given formula.
%Note that 
%\begin{align*}
%\mathrm{Prob}(G({m,n}) \le \lambda_1, \ldots, G(1,n) \le \lambda_n) &= \sum_{\mu \subset \lambda} \mathrm{Prob}(G({m,n}) = \mu_1, \ldots, G(1,n) = \mu_n)
%\end{align*}
%On the other hand by Theorem~\ref{main} we obtain
%\begin{align*}
%\sum_{\mu \subset \lambda} \mathrm{Prob}(G({m,n}) = \mu_1, \ldots, G(1,n) = \mu_n) &=  \prod_{i = 1}^{n}(1 - q_i)^m\, \sum_{\mu \subset \lambda} g_{\mu}(q_1, \ldots, q_n)\\
%&=  \prod_{i = 1}^{n}(1 - q_i)^m\, g_{\lambda}(1,q_1, \ldots, q_n)
%\end{align*}
%as needed.
\end{proof}
Using Jacobi-Trudi-type determinantal identities for $g_{\lambda}$ (see \cite{dy}) we get the next formulas. 
\begin{corollary}\label{det}
The following formulas hold \footnote{Here $e_{n}$ is the elementary symmetric polynomials, $h_n$ is the complete homogeneous symmetric polynomials, $1^m = (1, \ldots, 1)$ repeated $m$ times, and $\lambda'$ is the conjugate partition of $\lambda$.}
\begin{align*}
\mathrm{Prob}(G({m,n}) = \lambda_1, \ldots, G(1,n) = \lambda_{m}) &= \prod_{i = 1}^{n}(1 - q_i)^{m} \det[e_{\lambda'_i - i + j}(1^{\lambda'_i - 1}, q_1, \ldots, q_n)]_{i,j = 1}^{\lambda_1}, \\
&=\prod_{i = 1}^{n}(1 - q_i)^{m} \det[h_{\lambda_i - i + j}(1^{i - 1}, q_1, \ldots, q_n)]_{i,j=1}^{m}.
\end{align*}
\end{corollary}

%Consider now single point distributions.
\begin{corollary}[Single point distributions via Schur polynomials]\label{schur}
We have 
$$
\mathrm{Prob}(G(m,n) \le a) = \prod_{i = 1}^{n} (1 - q_i)^{m}\, s_{(a^m)}(1^m, q_1, \ldots, q_n)
$$
\end{corollary} 
\begin{proof}
From Corollary~\ref{le} and the first determinantal formula in Corollary~\ref{det} which coincides with the Jacobi-Trudi determinant for $s_{(a^m)}(1^m,q_1, \ldots, q_n)$ %Theorem~\ref{main} and using the branching formulas for $g_{\lambda}$ 
we have 
$$
\mathrm{Prob}(G(m,n) \le a) =\prod_{i = 1}^{n} (1 - q_i)^{m}\, g_{(a^m)}(1,q_1,\ldots,q_n) %= \sum_{\lambda \in \mathcal{P}_m\, :\, \lambda_1 \le a} g_{\lambda}(q_1, \ldots, q_n) 
= \prod_{i = 1}^{n} (1 - q_i)^{m}\, s_{(a^m)}(1^m, q_1, \ldots, q_n)
$$
as neeeded.
\end{proof}

%\begin{corollary}[Toeplitz determinant expressions]
%We have 
%%Let $k \in [1,M]$. The distribution of the part $\lambda_k$ can be %has the distribution 
%%expressed via the Toeplitz determinant
%\begin{align}\label{toep}
%%\mathrm{P}_{b,c,\mathbf{q}}(\pi \text{ has at most $a$ rows}) = 
%\mathrm{Prob}_{}(G(m,n) \le a) %= \frac{1}{Z_{b,c,\mathbf{q}}} s_{(a^b)}(1^b, q_1, \ldots, q_c) 
%%= \frac{1}{Z_{b,c,\mathbf{q}}} D_{a}(\varphi) 
%=  \prod_{i = 1}^{n}(1 - q_i)^m \det\left[\varphi_{i - j}\right]_{i,j=1}^{a},
%\end{align}
%where %symbol 
%$$\varphi_{k} = \sum_{\ell = 0}^{\infty} \binom{m}{\ell + k} %e_{\ell + k}(1^b) 
%e_{\ell}(q_1, \ldots, q_n)\quad \text{ or }\quad
% %\varphi(z) := 
% \sum_{k \in \mathbb{Z}} \varphi_k z^k = (1 + z)^{m} \prod_{i = 1}^n (1 + q_i z^{-1}).$$
%\end{corollary}

\begin{remark}
From this formula, via the Jacobi-Trudi identity one can also obtain Toeplitz as well as Fredholm determinantal expressions using the Borodin-Okounkov formula \cite{bo}.
\end{remark} 

\begin{remark}
These distribution formulas were presented in a special {\it iid} case $q_j = q$ in %, the result was presented in 
\cite{dy5}.
\end{remark}

\section{Generating series identities for $g_{\lambda}$}
By Theorem~\ref{main} we can define the probability distribution $\mathrm{P}_{m,n}$ on the set of integer partitions $\mathcal{P}_m$ by setting
$$
\mathrm{P}_{m,n}(\lambda) := \prod_{i = 1}^{n}(1 - q_i)^{m}\, g_{\lambda}(q_1, \ldots, q_n), \quad \lambda \in \mathcal{P}_{m}. 
$$
In particular, since $\sum_{\lambda \in \mathcal{P}_m} \mathrm{P}_{m,n}(\lambda) = 1$ we immediately obtain the following identity for dual Grothendieck polynomials (it can also be found in \cite{dy4, dy5}).
\begin{corollary}
The following identity holds %for dual Grothendieck polynomials
\begin{align*} %\label{gid}
 \sum_{\lambda \in \mathcal{P}_m} g_{\lambda}(q_1, \ldots, q_n) = \prod_{i = 1}^{n} (1 - q_i)^{-m}.
\end{align*}
\end{corollary}

Next, observe that we have the following {\it marginal distributions} for the parts $\lambda_k$:
$$
\mathrm{P}_{m,n}(\lambda_k \le a) = \mathrm{Prob}(G(m - k + 1, n) \le a)
$$
which give a {\it shift invariance} property
$$
\mathrm{P}_{m,n}(\lambda_k \le a) %= \mathrm{P}(G(M - i + 1, N) \le t) 
= \mathrm{P}_{m - k + 1,n}(\lambda_1 \le a)
$$
In particular, the last part $\lambda_m$ has distribution as %we have 
$\sum_{i = 1}^n W_i$ for independent $W_i$ geometrically distributed with parameter $q_i$.

\vspace{0.5em}

We now present a new more general identity for dual Grothendieck polynomials.

\begin{theorem}
Let $k \in [1, m]$ and $a \in \mathbb{N}$. The following identity holds 
$$
\sum_{\lambda\in \mathcal{P}_m,\, \lambda_k \le a} g_{\lambda}(q_1, \ldots, q_n) = \prod_{i = 1}^{n} (1 - q_i)^{1-k}\, s_{(a^{m - k + 1})}(1^{m - k + 1}, q_1, \ldots, q_n).
$$
%where %$s_{(\cdot)}$ is the Schur polynomial and 
%$1^{m} := (1, \ldots, 1)$ repeated $m$ times.
\end{theorem}

\begin{proof}
Recall that we have the marginal distributions %By the shift invariance property mentioned above we have 
$$
%\sum_{\lambda \in \mathcal{P}_m, \lambda_k \le a} \mathrm{P}_{m,n}(\lambda) = 
\mathrm{P}_{m,n}(\lambda_k \le a) %= \mathrm{P}_{m - k + 1, n}(\lambda_1 \le a) 
= \mathrm{Prob}(G(m-k+1, n) \le a)
$$
Using Corollary~\ref{schur}, we have
$$
\mathrm{Prob}(G(m-k+1, n) \le a) = \prod_{i = 1}^{n} (1 - q_i)^{m-k+1} s_{(a^{m - k + 1})}(1^{m - k + 1}, q_1, \ldots, q_n).
$$
On the other hand, by definition of the distribution $\mathrm{P}_{m,n}$ above, we get
$$
\mathrm{P}_{m,n}(\lambda_k \le a) = \prod_{i = 1}^n (1 - q_i)^m\, \sum_{\lambda \in \mathcal{P}_m,\, \lambda_k \le a} g_{\lambda}(q_1, \ldots, q_n)
$$
Combining the last two identities we obtain the needed.
\end{proof}

\begin{corollary}
For $k = m$ we obtain the following identity
$$
\sum_{\lambda\in \mathcal{P}_m,\, \lambda_m = a} g_{\lambda}(q_1, \ldots, q_n) = \prod_{i = 1}^{n} (1 - q_i)^{1-m}\, {h_{a}(q_1, \ldots, q_n)}{}
$$
%where $h_r$ is the complete homogeneous symmetric polynomial.
\end{corollary}
\begin{proof}
For $k = m$ we have 
$$
\sum_{\lambda\in \mathcal{P}_m,\, \lambda_m \le a} g_{\lambda}(q_1, \ldots, q_n) = \prod_{i = 1}^{n} (1 - q_i)^{1-m}\, {h_{a}(1,q_1, \ldots, q_n)}{}.
$$
Now the following recurrence relation for the polynomials $h$,
$$
h_{a}(1,q_1, \ldots, q_n) - h_{a-1}(1,q_1, \ldots, q_n) = h_{a}(q_1, \ldots, q_n)
$$
then gives the needed identity.
\end{proof}

\begin{remark}
There is one more connection of  dual Grothendieck polynomials with the corner growth model via {\it positive specializations} of $\{g_{\lambda}\}$, presented in \cite{dy3}. The distribution $\mathrm{P}_{m,n}$ can also be extended for any positive specialization as we discuss it for a special example in the next section. %Yet another relationship 
\end{remark}

\section{Plancherel limit and longest increasing subsequences}
Consider the specialization $q_i = {\gamma}/{n}$ for all $i \in [1,n]$, $\gamma > 0$ and let $n \to \infty$. We obtain
$$
%\mathrm{P}^{\, t}_{g,M}(\lambda) := 
\lim_{n \to \infty} \mathrm{P}_{m,n}(\lambda)= e^{-m \gamma} \lim_{n \to \infty} g_{\lambda}(\underbrace{\gamma/n, \ldots, \gamma/n}_{n \text{ times}}) = e^{-m \gamma} \sum_{n} \mathrm{P}_{\mathrm{gpl}, m, n}(\lambda) \frac{(m \gamma)^n}{n!},
$$
where $\mathrm{P}_{\mathrm{gpl}, m, n}(\lambda)$ is a probability distribution on the set of partitions $\lambda \subset (n^m)$, defined below.
%To describe it simpler, 

To define it, we need a generalization of standard Young tableaux. 
A plane partition $\pi$ is called a {\it strict tableau} (ST) if for some $n$, each entry $i \in [n] := \{1, \ldots, n\} $ appears in exactly one column of $\pi$. We then say that $[n]$ is a {\it content} of $\pi$. Let $\mathrm{ST}(\lambda, n)$ be the set of ST of shape $\lambda$ with content $[n]$ and $f_{\lambda}(n) = |\mathrm{ST}(\lambda, n)|$.

\begin{lemma}[\cite{dy4}]
We have 
$$\mathrm{P}_{\mathrm{gpl}, m,n}(\lambda) = \frac{f_{\lambda}(n)}{m^n}$$ is a well-defined probability measure on the set of integer partitions $\lambda \subset (n^m)$. %, i.e. 
%$$
%\sum_{\lambda \subset (a^b)} \frac{f_{\lambda}(a)}{b^a} = 1.
%$$
\end{lemma}

%Random words. 
Let $W_{n,m}$ be the set of words of length $n$ in the alphabet $[m]$. 
For a word $w  = w_1 \cdots w_n \in W_{m,n}$, a {\it weakly increasing subsequence} is a sequence of the form
$$
w_{i_1} \le \cdots \le w_{i_k}, \quad 1 \le i_1 < \cdots < i_k \le n,
$$
where $k$ is its length. Let $L_{i}(w)$ be the length of the longest weakly increasing subsequence of $w$ using the letters $\{m-i+1,\ldots, m \}$. In particular, $L_{1}(w)$ is just the number of $m$'s in $w$ and $L_{m}(w)$ is the length of the longest weakly increasing subsequence of $w$. 

\vspace{0.5em}

Consider the uniform probability measure on $W_{n,m}$. Then we have the following analogue of Theorem~\ref{main} in this Plancherel limit regime.

\begin{theorem}
We have
$$
\mathrm{Prob}(L_{m} = \lambda_1, \ldots, L_{1} = \lambda_m) = \mathrm{P}_{\mathrm{gpl}, m, n}(\lambda).
$$
\end{theorem}

A combinatorial version of this result (an analogue of Green's theorem for RSK), which can be turned into this statement, is proved in \cite{dy4}.
The distribution of $L_m$ was studied in \cite{tw} from which we obtain that for fixed $m$ the limiting distribution of the first row $\lambda_1$ satisfies
$$
\lim_{n \to \infty} \mathrm{P}_{\mathrm{gpl},m,n}\left(\frac{\lambda_1 - n/m}{\sqrt{2n/m}} \le t \right) = \mathrm{P}_{\mathrm{GUE}^0_{m}} (\lambda_{\max} \le t),
$$
where the r.h.s. is the distribution of the largest eigenvalue in $m \times m$  traceless Gaussian unitary ensemble (GUE).
In addition, note that $\lambda_m$ has binomial distribution with parameters $n$ and $1/m$ and hence after proper scaling it converges to normal distribution. Now, what is limiting joint distribution of the properly scaled shape $\lambda$ when $m$ is fixed? This would compare to the results in \cite{joh3} on limiting distribution of the shape for a random word under the RSK correspondence which converges to the spectrum of traceless GUE. 
\section*{Acknowledgements}

I am grateful to Askar Dzhumadil'daev, Igor Pak, Pavlo Pylyavskyy, and Leonid Petrov for many helpful conversations. 

%\newpage

\end{document}